\documentclass{amsart}

\usepackage[english]{babel}
\usepackage{amsfonts, amsmath, amsthm, amssymb,amscd,indentfirst,mathrsfs}

\usepackage{color}
\usepackage{xcolor}
\usepackage{MnSymbol}
\usepackage{tikz}
\usepackage{esint}

\usepackage[utf8]{inputenc} 
\usepackage[T1]{fontenc}

\usepackage{graphicx}
\usepackage{epstopdf}
\usepackage{latexsym}

\usepackage{palatino}
\usepackage{marginnote}

\newtheorem{theorem}{Theorem}[section]

\newtheorem{proposition}[theorem]{Proposition}

\newtheorem{definition}[theorem]{Definition}

\newtheorem{corollary}[theorem]{Corollary}

\newtheorem{remark}[theorem]{Remark}

\def\bp{\begin{proof}}
\def\ep{\end{proof}}

\def\Jc{{\mathcal J}}

\def\Hc{{\mathcal H}}

\def\Lc{{\mathcal L}}

\def\Kc{{\mathcal K}}

\def\Mc{{\mathcal M}}

\def\Nc{{\mathcal N}}

\def\Ac{{\mathcal A}}

\def\Ic{\mathcal{I}}

\begin{document}

\title[[An Alexandrov-Fenchel-type inequality]{An Alexandrov-Fenchel-type inequality in hyperbolic space with an application to  a Penrose inequality}

\author{Levi Lopes de Lima}
\address{Universidade Federal do Cear\'a (UFC),
Departamento de Matem\'{a}tica, Campus do Pici, Av. Humberto Monte, s/n, Bloco 914, 60455-760,
Fortaleza, CE, Brazil.}
\email{levi@mat.ufc.br}
\author{Frederico Gir\~ao}
\address{Universidade Federal do Cear\'a (UFC),
	Departamento de Matem\'{a}tica, Campus do Pici, Av. Humberto Monte, s/n, Bloco 914, 60455-760,
	Fortaleza, CE, Brazil.}
\email{fred@mat.ufc.br}
\thanks{The first author was partially supported by a CNPq/Brazil research grant}

\begin{abstract}
We prove a sharp Alexandrov-Fenchel-type inequality for star-shaped, strictly mean convex hypersurfaces in hyperbolic $n$-space, $n\geq 3$. The argument uses two new monotone quantities along the inverse mean curvature flow. As an application we establish, in any dimension, an optimal Penrose inequality for asymptotically hyperbolic graphs carrying a minimal horizon, with the equality occurring if and  only if the graph is an anti-de Sitter-Schwarzschild solution.
This sharpens previous results by Dahl-Gicquaud-Sakovich and settles, for this class of initial data sets, the conjectured Penrose inequality for time-symmetric space-times with negative cosmological constant. We also explain how our methods  can be easily adapted to derive an optimal Penrose inequality for asymptotically locally hyperbolic graphs in any dimension $n\geq 3$. When the horizon has the topology of a compact surface of genus at least one, this provides an affirmative answer, for this class of initial data sets, to a question posed by Gibbons, Chru\'sciel and Simon on the validity of a Penrose-type inequality for exotic black holes.  
\end{abstract}

\maketitle
\tableofcontents



\section{Introduction}

If $\Sigma\subset \mathbb R^n$ is a convex hypersurface
then the  Alexandrov-Fenchel  inequalities say that
\begin{equation}\label{af}
	\int_\Sigma\sigma_k(\kappa)d\Sigma\geq C_{n,k}\left(\int_{\Sigma}\sigma_{k-1}(\kappa)d\Sigma\right)^{\frac{n-k-1}{n-k}},
\end{equation}
where $\sigma_k(\kappa)$, $1\leq k\leq n-1$, is the $k^{\rm th}$ elementary symmetric function of the principal curvature vector $\kappa=(\kappa_1,\cdots,\kappa_{n-1})$ of $\Sigma$ and $C_{n,k}>0$ is a universal constant. Moreover, the equality holds in (\ref{af}) if and only if $\Sigma$ is a round sphere.
Classically, (\ref{af}) follows from the general theory of mixed volumes, so that convexity is used in an essential way; see \cite{S}. Recently, however, Guan and Li \cite{GL} used a suitable normalization of a certain inverse curvature flow to extend the validity of (\ref{af}), with  the corresponding rigidity statement, for any $\Sigma$ which is star-shaped and $k$-convex (which means that $\sigma_i(\kappa)\geq 0$ for $i=1,\cdots,k$).

An interesting question is to establish versions of these inequalities for appropriate classes of hypersurfaces in more general ambient manifolds, preferably with a corresponding rigidity statement for the case of equality. Here we focus on the case $k=1$
of
(\ref{af}), namely,
\begin{equation}\label{af1}
	c_n\int_\Sigma H d\Sigma\geq \frac{1}{2}\left(\frac{A}{\omega_{n-1}}\right)^{\frac{n-2}{n-1}},
\end{equation}
where $A$ is the area, $H=\sigma_1(\kappa)$ is the mean curvature,
$$
c_n=\frac{1}{2(n-1)\omega_{n-1}},
$$
and $\omega_{n-1}$ is the area of the unit sphere $\mathbb S^{n-1}\subset \mathbb R^n$. We take a first step to\-ward solving this problem by establishing a natural analogue of (\ref{af1}) for star-shaped, strictly mean convex hypersurfaces in hyperbolic $n$-space, $n\geq 3$; see Theorem \ref{main}.
The proof is partly inspired by \cite{GL} and uses two new monotone quantities for the inverse mean curvature flow in hyperbolic space. The precise asymptotics for this flow, which is a key ingredient  in our analysis, has been recently established by Gerhard \cite{G2} \cite{G3}; see \cite{D} for previous work on this subject.
Also, a Heintze-Karcher-type inequality due to Brendle \cite{B}  plays a key role in our proof. We also make use of a special case of a sharp geometric inequality by Brendle-Hung-Wang \cite{BHW}. We  note that Gallego and Solanes \cite{GS} proved related isoperimetric inequalities using integral-geometric methods, but their results do not seem to be sharp.

The inequality (\ref{af1}) has recently become relevant in the context of the Penrose inequality for asymptotically flat graphs carrying a minimal horizon \cite{L} \cite{dLG1}  and for asymptotically hyperbolic graphs carrying a constant mean curvature horizon \cite{dLG2}.  As an application of Theorem \ref{main} we establish an optimal Penrose inequality for asymptotically hyperbolic graphs carrying a minimal horizon, including
the rigidity statement according to which the equality holds only if $(M,g)$ is the graph realization of an anti-de Sitter-Schwarzs\-child solution; see Theorem \ref{main2}.
This Penrose inequality  improves  recent results by Dahl-Gicquaud-Sakovich \cite{DGS} and settles, for this class of initial data sets, the conjectured Penrose inequality for time-symmetric space-times with negative cosmological constant \cite{BC} \cite{Ma}.
We remark that the proof of the rigidity statement follows from the arguments in a  recent paper by Huang and Wu \cite{HW1}, as adapted
to the asymptotically hyperbolic case in  \cite[Section 5]{DGS}; see also \cite{dLG3}.

To explain our results, let us consider the hyperbolic $n$-space $\mathbb H^n$ with coordinates $(r,\theta)\in \mathbb R^+\times\mathbb S^{n-1}$ and endowed with the metric
\begin{equation}\label{modelr}
	g_1=dr^2+\sinh^2r\, h,
\end{equation}
where $r$ is the geodesic distance to a chosen origin corresponding to $r=0$ and $h$ is the round metric on $\mathbb S^{n-1}$. We say that a closed, embedded hypersurface $\Sigma\subset \mathbb H^n$
is {\em star-shaped} if it can be written as a radial graph over a geodesic sphere centered at the origin. Also, it is {\em strictly mean convex} if its mean curvature $H$ is positive everywhere.
We also consider $\rho_1:\mathbb H^n\to \mathbb R$,
\begin{equation}\label{rho}
	\rho_1(r)=\cosh r.
\end{equation}
With this notation at hand we can state the hyperbolic Alexandrov-Fenchel-type  inequality.

\begin{theorem}\label{main}
	If $\Sigma\subset \mathbb H^n$ is a star-shaped and strictly mean convex  hypersurface  then
	\begin{equation}\label{afh}
		c_n\int_\Sigma \rho_1 H d\Sigma\geq \frac{1}{2}\left(\left(\frac{A}{\omega_{n-1}}\right)^{\frac{n-2}{n-1}}+
		\left(\frac{A}{\omega_{n-1}}\right)^{\frac{n}{n-1}}\right),
	\end{equation}
	where $A$ is the area of $\Sigma$. Moreover, the equality holds if and only if $\Sigma$ is a geodesic sphere centered at the origin.
\end{theorem}

We now explain the relevance of this result for a certain Penrose inequality.
Recall that a Riemannian manifold $(M^n,g)$ is said to be {\em asymptotically hyperbolic} (AH) if there exists a compact subset $K\subset M$ and a diffeomorphism $\Psi:M-K\to \mathbb H^n-K_0$, where $K_0\subset \mathbb H^n$ is compact, such that
\begin{equation}\label{asympmani}
	\|\Psi_*g-g_1\|_{g_1}=O(e^{-\sigma r}),\quad \|\nabla_{g_1}\Psi_*{g}\|_{g_1}=O(e^{-\sigma  r}),
\end{equation}
as $r\to +\infty$, for some $\sigma>n/2$. We also assume that the difference between scalar curvatures, namely,
$$
\mathfrak R_g=R_{g}+n(n-1),
$$
is such that $\rho_1\mathfrak R_g$ is integrable. For any chart at infinity as in (\ref{asympmani}) it is possible to associate  a mass-like invariant 
$\mathfrak m_{\Psi}$ which lies in $\mathbb L^{n+1}$, the Lorentzian space endowed with the metric
\begin{equation}\label{lorinner}
	(z,w)=z_0w_0-z_1w_1-\cdots -z_nw_n;
\end{equation}
see Section \ref{masstype} for more details on this construction.
It turns out that the causal character of $\mathfrak m_{\Psi}$ is invariant under coordinate changes at infinity.  Moreover, 
the numerical invariant $\mathfrak m_{(M,g)}$ defined by
\begin{equation}\label{mass-likeinv}
	\mathfrak m_{(M,g)}^2=|(\mathfrak m_{\Psi},\mathfrak m_{\Psi})|
\end{equation}
does {\em not} depend on the chart $\Psi$ and is termed the {\em mass} of $(M,g)$. It is natural to choose $\mathfrak m_{(M,g)}>0$ if $\mathfrak m_{\Psi}$ is time-like and future directed.

The Positive Mass Conjecture in this context asserts that if $\mathfrak R_g\geq 0$ then $\mathfrak m_{\Psi}$ is time-like and future-directed or vanishes, the latter occurring only if $(M,g)$ is isometric to $(\mathbb H^n,g_1)$. Equivalently, $\mathfrak m_{(M,g)}\geq 0$ with equality holding only for hyperbolic space. This has been proved for the spin case by Chru\'sciel and Herzlich \cite{CH}, generalizing a previous contribution by Wang \cite{W}; see also \cite{ACG} for a similar result in low dimensions with the spin condition removed. Moreover, if $M$ carries a (possibly disconnected) compact, outermost minimal boundary $\Gamma$ (a {\em horizon}) of area $A$, then the corresponding Penrose Conjecture says that
\begin{equation}\label{pineq}
	\mathfrak m_{(M,g)}\geq\frac{1}{2}\left(\left(\frac{A}{\omega_{n-1}}\right)^{\frac{n-2}{n-1}}+
	\left(\frac{A}{\omega_{n-1}}\right)^{\frac{n}{n-1}}\right),
\end{equation}
with equality holding only if $(M,g)$ is the (exterior) anti-de Sitter-Schwarz\-schild solution. We refer  to Section \ref{masstype} and the surveys \cite{BC} and \cite{Ma} for background on this conjecture.

Progress in establishing (\ref{pineq})   has been restricted so far to the case of graphs, as we now pass to explain. Recall that the metric
\begin{equation}\label{model}
	\overline g_1=\rho_1^2d\tau^2+g_1, \quad \tau\in \mathbb R,
\end{equation}
realizes $\mathbb H^n\times \mathbb R$ as the hyperbolic $(n+1)$-space $\mathbb H^{n+1}$. Using this model we then say that a complete immersed hypersurface $M\subset \mathbb H^{n+1}$ is {\em asymptotically hyperbolic} (AH) if there exists a compact subset $K\subset M$ such that $M-K$ can be written as a vertical graph associated to a smooth function $u:\mathbb H^n-K_0\to \mathbb R$, where $K_0\subset \mathbb H^n$ is compact, so that (\ref{asympmani}) holds for the chart $\Psi$ given by $\Psi(x,u(x))=x$, $x\in M-K_0$; see Definition \ref{alshyp} below. As explained in \cite{dLG2}, if additionally $M$ carries a minimal horizon $\Gamma$ then we may
assume that $\mathfrak m_{\Psi}$ is time-like and future oriented so that after composing $\Psi$ with an isometry we have
\begin{equation}\label{bal}
	\mathfrak m_{(M,g)}=\mathfrak m_{\Psi}(\rho_1),
\end{equation}
where here we use that $\rho_1=z_0|_{\mathbb H^n}$ if we view $\mathbb H^n\subset \mathbb L^{n+1}$ as the standard hyperboloid.
Charts with this property are called {\em balanced}.
Now let $M$ be {\em balanced} in the sense that nonparametric  coordinates at infinity are balanced as above.
Moreover, assume that $\Gamma$ lies on a totally geodesic hypersurface $P\subset \mathbb H^{n+1}$ defined by $\tau=\tau_0$, $\tau_0\in\mathbb R$, and that $M$ meets $P$ orthogonally along $\Gamma$, so that $\Gamma$ is minimal (hence, a horizon indeed). Under these conditions and starting from (\ref{bal}) it is shown in \cite{dLG2} that
\begin{equation}\label{massformulae}
	\mathfrak m_{(M,g)}=c_n\int_M\Theta\mathfrak R_gdM+c_n\int_{\Gamma}\rho_1 H d\Gamma,
\end{equation}
where $\Theta=\langle N,\partial/\partial t\rangle$, with $N$ being the unit normal to $M$ pointing upward at infinity  and $H$ being the mean curvature of $\Gamma\subset P$ with respect to its inward pointing unit normal; compare to the more general formula in (\ref{massform}). 

We remark that if $M$ is a  graph then (\ref{massformulae}) has been previously proved in \cite{DGS}. If this is the case, so that $\Theta >0$, and if we assume further that $\mathfrak R_g\geq 0$ then we obtain
from (\ref{massformulae}) that
\begin{equation}\label{estima}
	\mathfrak m_{(M,g)}\geq c_n\int_{\Gamma}\rho H d\Gamma.
\end{equation}
In \cite{DGS} this estimate is used to obtain several sub-optimal versions of (\ref{pineq}).
For instance, assuming that $\Gamma\subset P=\mathbb H^n$ is $h$-convex (in the sense that all principal curvatures are at least $1$) and encloses the origin of $P$, the authors show that
$$
\mathfrak m_{(M,g)}\geq \frac{1}{2}\left(\left(\frac{A}{\omega_{n-1}}\right)^{\frac{n-2}{n-1}}+\sinh r_{\rm in}\frac{A}{\omega_{n-1}}\right),
$$
where $r_{\rm in}$ is the radius of the largest geodesic ball centered at the origin and contained in the region enclosed by $\Gamma$. Notice that this only yields the conjectured inequality (\ref{pineq}) if $\Gamma$ is a geodesic sphere centered at the origin.
In view of Theorem \ref{main},  however, if we take $\Sigma=\Gamma$ we  immediately obtain the first statement in the following theorem; a more general result, covering the locally hyperbolic case, can be found in Theorem \ref{theomain} below. Recall that a hypersurface is said to be {\em mean convex} if its mean curvature is non-negative everywhere.

\begin{theorem}\label{main2}
	Let $(M,g)\subset \mathbb H^{n+1}$ be a  balanced AH graph carrying a minimal horizon $\Gamma$ as above. If we assume further that $\Gamma\subset P=\mathbb H^n$ is star-shaped (with respect to the origin) and mean convex then
	(\ref{pineq}) holds if $R_g\geq -n(n-1)$. Moreover, the equality occurs if and only if $(M,g)$ is the graph realization of an (exterior) anti-de Sitter-Schwarzschild solution (which is obtained by  taking $\epsilon=1$ in (\ref{embed}) below).
\end{theorem}

As remarked above, the rigidity statement requires a separate argument and is based on results in a recent preprint by Huang and Wu \cite{HW1}.

This paper is organized as follows. In Section \ref{masstype}
we recall the definitions and main properties of mass-type invariants for asymptotically locally hyperbolic manifolds. 
The proofs of Theorems \ref{main} and \ref{main2}, which 
use the material on the inverse mean curvature flow discussed in Section \ref{imcfsec} and in Appendix \ref{refined},
are presented in Section \ref{proofmain}. 
We observe that the Penrose inequality (\ref{pineq}) admits a natural  generalization to the case in which the geometry at infinity is asymptotically {\em locally} hyperbolic. In Appendix  \ref{alhcaseap} we indicate how the arguments leading to the proofs of Theorems \ref{main} and \ref{main2} can be adapted to prove a sharp Penrose-type inequality for asymptotically locally hyperbolic graphs, which is described in Theorem \ref{theomain} below.

\vspace{0.2cm}
\noindent
{\bf Acknowledgements.} The work leading to this paper started when the authors were visiting IMPA on the occasion of its 2012 Summer Program and they would like to thank Prof.  F. Marques for providing a wonderful scientific environment. Also, the authors would like to thank Y. Ge,  P. K. Hung, A. Neves, G. Wang,  M. T. Wang and Y. Wei for discussions and  C. Xiong for  pointing out an error in a previous version of this work. Finally, we thank the anonymous referees for suggestions that contributed to substantially improve the presentation. 	

\section{Mass-type invariants and Penrose inequalities for asymptotically locally hyperbolic manifolds}
\label{masstype}

In this section we review the main properties of mass-type invariants for asymptotically hyperbolic manifolds;
see  \cite{CH}, \cite{CN}, \cite{He} and \cite{M} for more details. The class of invariants presented here, which appear in natural generalizations of the classical Penrose inequality, is just a special case of a much more general construction due to Chru\'sciel, Herzlich and Nagy \cite{CH} \cite{CN} \cite{He}. Their definition applies in particular to the class of asymptotically locally hyperbolic manifolds we consider here.

We start by describing the corresponding {\em locally hyperbolic} (LH) reference metrics. Fix $\epsilon =0,\pm 1$ and let $(N^{n-1},h)$ be a closed space form of sectional curvature $\epsilon$. In the product manifold $P_\epsilon=I_\epsilon\times N$, where $I_{-1}=(1,+\infty)$ and $I_0=I_1=(0,+\infty)$, define the metric
\begin{equation}\label{half}
	g_\epsilon=\frac{d\tilde r^2}{\rho_{\epsilon}(\tilde r)^2}+\tilde r^2h, \quad \tilde r\in I_\epsilon,
\end{equation}
where
\begin{equation}\label{half2}
	\rho_\epsilon (\tilde r)=\sqrt{\tilde r^2+\epsilon}.
\end{equation}
It is easy to check that $(P_\epsilon,g_\epsilon)$ is {locally hyperbolic} in the sense that its sectional curvature is constant and equal to $-1$.
We also note that the manifold
$Q_\epsilon=\mathbb R\times P_\epsilon$, endowed with the metric
\begin{equation}\label{aswe}
	\overline g_\epsilon=\rho_\epsilon^2d\tau^2+g_\epsilon,
\end{equation}
is locally hyperbolic as well.
Moreover, if $\tau_0\in\mathbb R$ then the horizontal slice $P_{\epsilon}^{\tau_0}\subset Q_\epsilon$ given by $\tau=\tau_0$ is totally geodesic. This follows from the fact that the vertical vector field $\partial_\tau$ is Killing. In particular, each $P_\epsilon^{\tau_0}$ can be naturally identified to $P_\epsilon$.
Notice that if we take $\epsilon=1$ and $(N,h)$ to be a round sphere then we recover hyperbolic space as in the Introduction; compare (\ref{half})-(\ref{half2}) with (\ref{modelr})-(\ref{rho}) after setting $\tilde r=\sinh r$.

Let $(M^n,g)$ be a complete $n$-dimensional Riemannian manifold, possibly carrying a compact inner boundary $\Gamma$. For simplicity we assume that $M$ has a unique end, say $E$. We say that $(M,g)$ is {\em asymptotically locally hyperbolic} (ALH) if there exists a chart $\Psi$ taking $E$ to the end of $P_\epsilon$ corresponding to $\tilde r=+\infty$
so that, as $\tilde r\to +\infty$,
\begin{equation}\label{alhcond}
	\|\Psi_*g-g_{\epsilon}\|_{g_\epsilon}+ \|d\Psi_*g\|_{g_\epsilon}=O\left(\tilde r^{-\sigma}\right),
\end{equation}
for some $\sigma>n/2$. We also assume that $\mathfrak R_g=R_{g}+n(n-1)$ is such that $\rho_\epsilon \mathfrak R_g$ is  integrable. If $\epsilon=1$ and $N=\mathbb S^{n-1}$ we say that $(M,g)$ is {\em asymptotically hyperbolic} (AH).

It turns out that each model space $P_\epsilon$ is {\em static} in the sense that the space
\begin{equation}
	\label{staticeq}
	\Nc_{g_\epsilon}=\{f\in C^{\infty}(P_\epsilon);(\Delta_{g_\epsilon}f)g_\epsilon-\nabla^2_{g_\epsilon}f+f{\rm Ric}_{g_\epsilon}=0\}
\end{equation}
is non-trivial, as we can easily check that $\rho_\epsilon\in\Nc_{g_\epsilon}$. 
If $\Psi$ is a chart at infinity as above, we define the corresponding mass functional $\mathfrak m_{\Psi}:\Nc_{g_\epsilon}\to \mathbb R$ by
\begin{equation}\label{mass}
	\mathfrak m_{\Psi}(\varphi)=\lim_{r\to +\infty}c_n\int_{N_r}\left(\varphi\left({\rm div}_{g_\epsilon}e-d{\rm tr}_{g_\epsilon}e\right)-
	i_{\nabla_{g_\epsilon}\varphi}e+({\rm tr}_{g_\epsilon}e)d\varphi\right)(\nu_{r})dN_{r},
\end{equation}
where $e=\Psi_*g-g_\epsilon$, $\nu_r$ is the unit normal to $N_r=\{r\}\times N$ and
\[
c_n=\frac{1}{2(n-1)\vartheta_{n-1}}, \quad \vartheta_{n-1}={\rm area}(N,h).
\]
If $\Phi$ is another chart at infinity one verifies that
\begin{equation}\label{appear}
	\mathfrak m_{\Phi}(\varphi)=\mathfrak m_{\Psi}(\varphi\circ \Ic^{-1}),
\end{equation}
where $\Ic\in {\rm Isom}(P_\epsilon)$ satisfies
$$
\|\Phi\circ\Psi^{-1}-\Ic\|_{g_\epsilon}=O(\tilde r^{-\sigma}).
$$
Thus, in order to get a numerical invariant out of this scheme we need a detailed description of the structure of the action of ${\rm Isom}(P_\epsilon)$ on $\Nc_{g_\epsilon}$ appearing on  the right-hand side of (\ref{appear}). 

We first consider the case $\epsilon=1$ and $N_1=\mathbb S^{n-1}$, so that $P_1=\mathbb H^n$. Here, $\Nc_{g_1}$ is generated by $\{z_i|_{\mathbb H^n}\}_{i=0}^n$, where we
view $\mathbb H^n\subset \mathbb L^{n+1}$, the Lorentz space endowed with the metric (\ref{lorinner}). 
Since the action of ${\rm Isom}(\mathbb H^n)$ on $\Nc_{g_1}=\mathbb L^{n+1}$  preserves 
(\ref{lorinner}), 
with  $\rho_1=z_0$ being time-like and future oriented, it follows that the real number $\mathfrak m_{(M,g)}$ defined up to a sign by
(\ref{mass-likeinv})
does {\em not} depend on the chart $\Psi$ and is termed the {\em mass} of $(M,g)$. We note  that the causal character of $\mathfrak m_{\Psi}$ is also invariant under coordinate changes at infinity, so it is natural to choose $\mathfrak m_{(M,g)}>0$ if $\mathfrak m_{\Psi}$ is time-like and future directed. As already observed in the Introduction, if $M$ carries a minimal horizon $\Gamma$ then we may assume, after composition with an isometry, that (\ref{bal}) holds. 

In contrast to the hyperbolic case described in the previous paragraph, it is known that for $\epsilon\leq 0$ the space $\Nc_{g_\epsilon}$ is one-dimensional, generated by $\rho_\epsilon$. Thus, in all the cases considered here  a mass-type numerical invariant 
$\mathfrak m_{(M,g)}$
is obtained by evaluating the right-hand side of (\ref{mass}) on $\varphi=\rho_\epsilon$.

An ALH manifold $(M,g)$ as above can be thought of as the initial data set of a time-symmetric solution of the Einstein field equations with negative cosmological constant. The invariant $\mathfrak m_{(M,g)}$ is then interpreted as the total mass of the solution. Physical reasoning predicts  that $\mathfrak m_{(M,g)}$ should have the appropriate sign under the relevant dominant energy condition, namely, 
$\mathfrak R_g\geq 0$ 
(equivalently, $R_g\geq -n(n-1)$). When $M$ carries a compact minimal horizon $\Gamma$ one expects the invariant to satisfy a Penrose-type inequality in the sense that  it should be bounded from below by a suitable expression involving the area $|\Gamma|$ of $\Gamma$. In order to figure out the correct form of this inequality, we consider the so-called {\em Kottler black hole metrics}, which are deformations of the LH metrics $g_\epsilon$ above.

Let us introduce a real parameter $m>0$
and consider the metric
\begin{equation}\label{bhmet}
	g_{\epsilon,m}=\frac{d\tilde r^2}{\rho_{\epsilon,m}(\tilde r)^2}+\tilde r^2h,
\end{equation}
where
\[
\rho_{\epsilon,m}(\tilde r)=\sqrt{\tilde r^2+\epsilon-\frac{2m}{\tilde r^{n-2}}}.
\]
For each $m$ as above, it is easy to see that  the function
\[
\tilde r\mapsto f_{\epsilon,m}(\tilde r)=\tilde r^n+\epsilon \tilde r^{n-2}-2m
\]
is strictly positive for  $\tilde r>\tilde r_{\epsilon,m}$, where $\tilde r_{\epsilon, m}$ is the unique positive zero of $f_{\epsilon,m}$. Thus, the metric $g_{\epsilon,m}$ is well defined on the product $P_{\epsilon,m}=I_{\epsilon,m}\times N$, where  $I_{\epsilon,m}=\{\tilde r;\tilde r>\tilde r_{\epsilon,m}\}$. Moreover, it extends smoothly to the slice $\tilde r=\tilde r_{\epsilon,m}$, the so-called {\em horizon}, denoted $\Hc_{\epsilon,m}$. This terminology can be justified as follows. The metric $g_{\epsilon,m}$ is {\em static} in the sense that
$\rho_{\epsilon,m}\in\Nc_{g_{\epsilon,m}}$.
It is well-known that this is equivalent to the assertion that  the Lorentzian metric
\[
\tilde g_{\epsilon,m}=-\rho_{m,\epsilon}^2d\tau^2+g_{\epsilon,m},
\]
defined on $Q_{\epsilon,m}=\mathbb R\times P_{\epsilon,m}$,
is a solution to the vacuum  field equations with negative cosmological constant:
\[
{\rm Ric}_{\tilde g_{\epsilon,m}}=-n\tilde g_{\epsilon,m}.
\]
Moreover, the null hypersurface $\tilde r=\tilde r_{\epsilon,m}$
defines the event horizon surrounding the central singularity $\tilde r=0$. This justifies the horizon terminology and explains why  $g_{\epsilon,m}$ is termed a black hole metric.

A computation shows  that
if $(\theta_1,\cdots,\theta_{n-1})$ are coordinates in $N_{\tilde r}$ then the sectional curvatures of $g_{\epsilon,m}$ are
\[
K(\partial_{\tilde r},\partial_{{\theta}_i})=-1-(n-2)\frac{m}{\tilde r^n}
\]
and
\[
K(\partial_{{\theta}_i},\partial_{{\theta}_j})=-1+\frac{2m}{\tilde r^{n-2}}.
\]
This not only shows that $g_{\epsilon,m}$ satisfies the appropriate dominant energy condition, since
its scalar curvature is
$
R_{g_{\epsilon,m}} =  -n(n-1),
$
but also suggests that $g_{\epsilon,m}$ is ALH in the sense described above. In fact, a straightforward computation  gives
\[
\|g_{\epsilon,m}-g_{\epsilon}\|_{g_\epsilon}+ \|dg_{\epsilon,m}\|_{g_\epsilon}=O\left(\tilde r^{-n}\right),
\]
as expected.
Using (\ref{mass}) with $\varphi=\rho_\epsilon$ we finally conclude that $\mathfrak m_{(P_{\epsilon,m},g_{\epsilon,m})}=m$, which shows that $m$ should be interpreted as the total mass of $g_{\epsilon,m}$.

One immediately finds that
the area  $|\Hc_{\epsilon,m}|$ of the horizon $\Hc_{\epsilon,m}$ of \linebreak $(P_{\epsilon,m},g_{\epsilon,m})$ relates to its mass  $m$ by means of
\[
m=\frac{1}{2}\left(\left(\frac{|\Hc_{\epsilon,m}|}{\vartheta_{n-1}}\right)^{\frac{n}{n-1}}
+\epsilon \left(\frac{|\Hc_{\epsilon,m}|}{\vartheta_{n-1}}\right)^{\frac{n-2}{n-1}}\right).
\]
Thus, in analogy with the standard Penrose inequality (\ref{pineq}), it is natural to conjecture  that if $(M,g)$ is an $n$-dimensional ALH manifold (with respect to the reference metric $g_\epsilon$) carrying an outermost minimal  horizon $\Gamma$ and satisfying $R_g\geq -n(n-1)$ everywhere then there holds
\begin{equation}\label{propcon}
	\mathfrak m_{(M,g)}\geq\frac{1}{2}\left(\left(\frac{|\Gamma|}{\vartheta_{n-1}}\right)^{\frac{n}{n-1}}
	+\epsilon \left(\frac{|\Gamma|}{\vartheta_{n-1}}\right)^{\frac{n-2}{n-1}}\right),
\end{equation}
with the equality occurring if and only if $g$ is isometric to the corresponding black hole metric.

\begin{remark}\label{physical}
	{\rm For $\epsilon\leq 0$ and in the physical dimension $n=3$, (\ref{propcon}) first appears in \cite{CS} as a conjectured inequality whose veracity would follow in case the use of the so-called Geroch's monotonicity of the Hawking mass under the inverse mean curvature flow, as envisaged by Gibbons \cite{G}, could be justified.
		Contrary to this rather optimistic initial expectation, Neves \cite{N} has shown that, at least in the AH case, the convergence properties of the flow at infinity are insufficient to implement Geroch's scheme. Similar remarks should also apply in the general ALH context, even though Lee and Neves \cite{LN} have recently established that Geroch's strategy works in the so-called \lq non-positive mass range\rq; see Remark \ref{negmass}.
		Despite these negative results, Theorem \ref{theomain} below confirms that our methods apply to handle the special case of graphs in {\em any} dimension $n\geq 3$.}
\end{remark}

To motivate our setting we observe that each $(P_{\epsilon,m},g_{\epsilon,m})$ can be isometrically immersed as a radially symmetric vertical graph inside $(Q_{\epsilon},\overline g_\epsilon)$: the defining function $u_{\epsilon,m}:I_{\epsilon,m}\to \mathbb R$ satisfies $u_{\epsilon,m}(\tilde r_{\epsilon,m})=0$ and
\begin{equation}\label{embed}
	\rho_{\epsilon}(\tilde r)^2\left(\frac{du_{\epsilon,m}}{d\tilde r^2}\right)^2
	=\frac{1}{\rho_{\epsilon,m}(\tilde r)^2}-\frac{1}{\rho_{\epsilon}(\tilde r)^2}
\end{equation}
It is clear from this that the horizon $\Hc_{\epsilon,m}$ lies on the totally geodesic horizontal slice $P_\epsilon^{0}$, with the intersection $M\cap P_\epsilon^{0}$ being orthogonal along $\Hc_{\epsilon,m}$. This motivates us to consider a more general class of hypersurfaces in $(Q_{\epsilon},\overline g_\epsilon)$.

\begin{definition}\label{alshyp}
	We say that a complete hypersurface $(M,g)\subset (Q_{\epsilon},\overline g_\epsilon)$, possibly carrying a compact inner boundary $\Gamma$, is {\em asymptotically locally hyperbolic} (ALH) if there exists a compact set $K\subset M$ so that $M-K$ can be written as a graph over the end $E_0$ of the horizontal slice
	$P_{\epsilon}^{0}\subset Q_{\epsilon}$, with the graph being  associated to a smooth function $u$ such
	that the asymptotic condition (\ref{alhcond}) holds for the nonparametric chart $\Psi_u(x,u(x))=x$, $x\in E_0$.
	Moreover, we assume that $\rho_\epsilon \mathfrak R_{g}$ is integrable. As usual, if $\epsilon=1$ and $N=\mathbb S^{n-1}$ then we say that $(M,g)$ is {\em asymptotically hyperbolic} (AH). 
\end{definition}

Under these conditions, the mass $\mathfrak m_{(M,g)}$ can be computed by taking $\Psi=\Psi_u$ in (\ref{mass}); as in the Introduction we assume that $\Psi_u$ is balanced if $\epsilon=1$. More precisely, if  the inner boundary $\Gamma$ lies on some totally geodesic, horizontal slice $P_{\epsilon}^{\tau_0}$, which we of course identify with $P_\epsilon$, and moreover that the intersection  $M\cap P^{\tau_0}_\epsilon$ is orthogonal along $\Gamma$, so that $\Gamma\subset M$ is minimal and hence a horizon indeed, then the computations in \cite{dLG2} actually give the following integral formula for the mass:
\begin{equation}\label{massform}
	\mathfrak m_{(M,g)}=c_n\int_M\Theta\left(R_g+n(n-1)\right) dM+
	c_n\int_\Gamma \rho_{\epsilon} Hd\Gamma,
\end{equation}
where $\Theta=\langle \partial/\partial t,N\rangle$, $N$ is the unit normal to $M$, which we choose so as to point upward at infinity, and $H$ is the mean curvature of $\Gamma\subset P_{\epsilon}$ with respect to its {\em inward} pointing unit normal, which means that the normal points in the direction {\em  opposite} to the end of $P_{\epsilon}$ given by $\tilde r=+\infty$. In particular, if $R_g\geq -n(n-1)$ and $M$ is a graph ($\Theta>0$) then
\begin{equation}\label{massineq}
	\mathfrak m_{(M,g)}\geq c_n\int_\Gamma \rho_{\epsilon} Hd\Gamma.
\end{equation}

We are now in a position to state the following Alexandrov-Fenchel-type inequality, which extends Theorem \ref{main} to the case $\epsilon\leq 0$.

\begin{theorem}\label{aftype}
	Let $\Sigma\subset P_\epsilon=P$ be a compact embedded hypersurface which is {\em star-shaped} in the sense that it can be written as a radial graph over a slice $N_{\tilde r}=\{\tilde r\}\times N$ and {\em strictly mean convex} in the sense that its mean curvature satisfies $H>0$. Then there holds
	\begin{equation}\label{aftype2}
		c_n\int_\Sigma \rho_{\epsilon} Hd\Sigma\geq \frac{1}{2}\left(
		\left(\frac{|\Sigma|}
		{\vartheta_{n-1}}\right)^{\frac{n}{n-1}}
		+\epsilon \left(\frac{|\Sigma|}{\vartheta_{n-1}}\right)^{\frac{n-2}{n-1}}\right),
	\end{equation}
	with the equality occurring if and only if $\Sigma$ is a slice.
\end{theorem}

By making $\Sigma=\Gamma$ and combining (\ref{massineq}) and (\ref{aftype2}) we immediately obtain the following sharp Penrose-type inequality, which extends Theorem \ref{main2} to the case $\epsilon\leq 0$.

\begin{theorem}\label{theomain}
	If $M\subset Q_{0,\epsilon}$ is an ALH graph as above, so that its horizon $\Gamma\subset P^{\tau_0}_\epsilon$ is
	star-shaped and  mean convex in the sense that $H\geq 0$, then
	\begin{equation}\label{penrose}
		\mathfrak m_{(M,g)}\geq \frac{1}{2}\left(
		\left(\frac{|\Gamma|}{\vartheta_{n-1}}\right)^{\frac{n}{n-1}}
		+\epsilon\left(\frac{|\Gamma|}{\vartheta_{n-1}}\right)^{\frac{n-2}{n-1}}\right),
	\end{equation}
	with the equality holding if and only if $(M,g)$ is (congruent to) the graph realization (\ref{embed}) of the corresponding black hole solution.
\end{theorem}

\begin{remark}\label{effec}
	{\rm Under the conditions of Theorem \ref{theomain}, the mass $\mathfrak m_{(M,g)}$ is always positive due to (\ref{massineq}). In particular, if $\epsilon=-1$ the lower bounds (\ref{aftype2}) and (\ref{penrose}) only are  effective if we further assume that  $|\Gamma|>\vartheta_{n-1}$.}
\end{remark}

As already observed, the following corollary provides a positive answer to   a question posed by Gibbons \cite{G} and Chru\'sciel-Simon \cite{CS} for the class of initial data sets we consider.

\begin{corollary}\label{cormain}
	If the horizon $\Gamma$ is a surface of genus $\gamma\geq 1$
	then
	\begin{equation}\label{haw}
		\mathfrak m_{(M,g)}\geq \left(\frac{4\pi}{\vartheta_2}\right)^{{3}/{2}}
		\sqrt{\frac{|\Gamma|}{16\pi}}\left(1-\gamma+\frac{|\Gamma|}{4\pi}\right),
	\end{equation}
	where for $\gamma=1$ we assume the normalization $\vartheta_2=4\pi$.
	Moreover, the equality holds if and only if $(M,g)$ is (congruent to) the graph realization of the corresponding black hole solution.
\end{corollary}

\begin{proof}
	If $\gamma\geq 2$ this follows by taking $n=3$ and $\epsilon=-1$ in the theorem and  observing that Gauss-Bonnet gives $\vartheta_2=4\pi (\gamma-1)$. If $\gamma=1$ we take $\epsilon=0$ and use the normalization.
\end{proof}

The proofs of Theorems \ref{aftype} and \ref{theomain} are straightforward adaptations of the proofs of Theorems \ref{main} and \ref{main2}. The necessary modifications are briefly described in Appendix \ref{alhcaseap}. We note that further results in this direction have been obtained in \cite{GWWX}.

\begin{remark}\label{negmass}
	{\em As discussed in \cite{CS}, when $\epsilon=-1$ the Kottler metrics (\ref{bhmet}) also describe static black hole solutions when the parameter $m$ becomes negative up to a certain critical value, namely,
		\[
		m_{\rm crit}=-\frac{(n-2)^{\frac{n-2}{2}}}{n^{\frac{n}{2}}}.
		\]
		In this regard we mention that Lee and Neves \cite{LN} used the Huisken- \linebreak Ilmanen's formulation of the inverse mean curvature flow to establish a Penrose-type inequality for conformally compact ALH $3$-manifolds in this mass range. More precisely, they prove (\ref{haw}) with the mass replaced by the supremum of the mass aspect function, which is assumed to be non-positive along the boundary at infinity. In particular, their manifolds always have {\em non-positive} mass while our graphs necessarily satisfy $\mathfrak m_{(M,g)}>0$; see Remark \ref{effec}. Thus, their result and Corollary \ref{cormain} are in a sense complementary to each other.}
\end{remark}

\section{Geometric flows for hypersurfaces}\label{imcfsec}

As mentioned above,
the proof of Theorem \ref{main} uses the inverse mean curvature flow recently studied by Gerhardt \cite{G2} \cite{G3}; see also \cite{D}. As a preparation for the argument, let us start by considering  a closed, isometrically immersed hypersurface $\Sigma\subset\mathbb H^{n}$ with unit normal $\xi$.
As usual, we denote by $g_1$ both the standard metric on $\mathbb H^n$ and its restriction to $\Sigma$. Also, $b$ is the second fundamental form of $\Sigma$. For simplicity we set $D=\nabla_{g_1}$. Thus, if $X$ and $Y$ are vector fields tangent to $\Sigma$,
$$
b(X,Y)=g_1(aX,Y)=\langle aX,Y\rangle,
$$
where
$$
aX=-D_X\xi,
$$
is the shape operator.
As before, we denote by $\kappa=(\kappa_1,\cdots,\kappa_{n-1})$ the principal curvature vector of $\Sigma$, so that
\begin{equation}\label{meancurv}
	H=\sigma_1(\kappa)={\rm tr}_{g_1}b
\end{equation}
is the {\em mean curvature}.
It follows from the Cauchy-Schwarz inequality that
\begin{equation}\label{cauchy}
	(n-1)|a|^2\geq H^2,
\end{equation}
with the equality occurring at a given point if and only if $\Sigma$ is umbilical there.
We also consider the {\em extrinsic scalar curvature} of the immersion, namely,
\begin{equation}\label{extscalar}
	K=\sigma_2(\kappa)=\sum_{i<j}\kappa_i\kappa_j=\frac{1}{2}\left(H^2-|a|^2\right).
\end{equation}
Notice that these invariants are related by the {Newton-MacLaurin inequality}:
\begin{equation}\label{lp}
	2K\leq \frac{n-2}{n-1}H^2,
\end{equation}
with the equality holding at a given point only if $\Sigma$ is umbilical there \cite{HLP}.
We also recall the {\em support function}
\begin{equation}\label{suppdef}
	p=\langle D\rho,\xi\rangle,
\end{equation}
where we set $\rho=\rho_1$ for simplicity. This  relates to $\rho$ and $H$ by means of the following Minkowski identity:
\begin{equation}\label{laprho}
	\Delta\rho=(n-1)\rho+H p,
\end{equation}
where $\Delta={\rm div}\circ \nabla$ is the Laplacian of $g_1|_\Sigma$. This is a consequence of the fact that the vector field $D\rho$ is {\em conformal}, that is,
\begin{equation}\label{conf}
	D_XD\rho=\rho X,
\end{equation}
for any vector field $X$ on  $\mathbb H^n$. Another useful consequence of (\ref{conf}) is the formula
\begin{equation}\label{mink22}
	{\rm div}\left(G\nabla\rho\right)=(n-2)\rho H+2pK,
\end{equation}
where
\begin{equation}\label{newttt}
	G=HI-a
\end{equation}
is the {\em Newton tensor} of $a$;
see \cite{AdLM} for further details.

We now consider an one-parameter family $X(t,\cdot):\Sigma^{n-1}\to \mathbb H^{n}$, $t\in [0,\epsilon)$, of closed, isometrically immersed hypersurfaces evolving  according to
\begin{equation}\label{evol}
	\frac{\partial X}{\partial t}=F\xi,
\end{equation}
where $\xi$ is the unit normal to $\Sigma_t=X(t,\cdot)$ and $F$ is a general speed function. To save notation we also denote the evolving hypersurface simply by $\Sigma$ whenever no confusion arises.
The following evolution equations are well-known \cite{Z}.

\begin{proposition}\label{evolequations}
	Under the flow (\ref{evol}) we have:
	\begin{enumerate}
		\item The unit normal evolves as
		\begin{equation}\label{unit}
			\frac{\partial \xi}{\partial t}=-\nabla F;
		\end{equation}
		\item
		The area element $d\Sigma$ evolves as
		\begin{equation}\label{evolvelemvol}
			\frac{\partial}{\partial t}d\Sigma=-FHd\Sigma.
		\end{equation}
		In particular, if $A$ is the area of $\Sigma$ then
		\begin{equation}\label{evolvarea}
			\frac{dA}{dt}=-\int_\Sigma FH d\Sigma;
		\end{equation}
		\item The mean curvature evolves as
		\begin{equation}\label{evolvmean}
			\frac{\partial H}{\partial t}=\Delta F+(|a|^2-(n-1))F.
		\end{equation}
	\end{enumerate}
\end{proposition}

If $\Sigma$ is star-shaped and mean convex then our conventions imply that $\xi$ is the {\em inward} pointing unit normal vector.
Thus, in the model (\ref{modelr}), $\Sigma$ can be graphically represented by means of a  map of the type
\begin{equation}\label{graph}
	\theta\in\mathbb S^{n-1}\mapsto (u(\theta),\theta)\in \mathbb H^n,
\end{equation}
for some smooth function $u$. In particular, if $\theta=(\theta_1,\cdots,\theta_{n-1})$ is a local coordinate system on $\mathbb S^{n-1}$ and $E_i=\partial/\partial \theta_i$ then the tangent space to the graph is spanned by
\begin{equation}\label{work1}
	Z_i=u_i\frac{\partial}{\partial r}+E_i,\quad u_i=E_i(u),\quad i=1,\cdots,n-1,
\end{equation}
so we can take
\begin{equation}\label{work2}
	\xi=\frac{1}{W}\left(\frac{u^i}{\dot\rho(u)^2}E_i-\frac{\partial}{\partial r}\right), \quad W=\sqrt{1+|\nabla v|_h^2},
\end{equation}
where
\begin{equation}\label{extrashi}
	v=\varphi(u), \quad \dot\varphi=1/\dot\rho,
\end{equation}
with
\[
\dot\rho(u)=\sinh u;
\]
see \cite{G2} or \cite{D}.
Also,
\begin{equation}\label{supportform}
	p=-\frac{\sinh u}{W}.
\end{equation}
Notice that $p\leq 0$.

From now on we assume that $\Sigma=\Sigma_t=X(t,\cdot)$ is a one-parameter family of star-shaped, strictly mean convex hypersurfaces evolving according to (\ref{evol}). This assumption will be justified later on for the flows  we shall consider; see Proposition  \ref{coll} and Remark \ref{brendlef}.

\begin{proposition}\label{evolvsupport}
	Under the above conditions, the function $\rho$ evolves along the flow (\ref{evol}) according to
	\begin{equation}\label{evolvrho}
		\frac{\partial \rho}{\partial t}=p F.
	\end{equation}
\end{proposition}

\begin{proof}
	As noted above, we can graphically represent $\Sigma$ by (\ref{graph}), where $u$ is time dependent, so that
	(\ref{evol}) implies
	$$
	\frac{\partial u}{\partial t}=-\frac{F}{W}.
	$$
	Since $u=r$ along $\Sigma$ we have
	\begin{eqnarray*}
		\frac{\partial\rho}{\partial t}=\sinh u\frac{\partial u}{\partial t}=-\frac{\sinh u\, F}{W},
	\end{eqnarray*}
	and the result follows from (\ref{supportform}).
\end{proof}

The following proposition computes the variation of the curvature integral
\begin{equation}\label{alpha}
	\Ic(\Sigma)=\int_\Sigma\rho Hd\Sigma
\end{equation}
on the left-hand side of (\ref{afh}).

\begin{proposition}\label{evolvfunc}
	Along the flow (\ref{evol}) we have
	\begin{equation}\label{evolrhoH}
		\frac{d\Ic}{dt}=2\int_\Sigma pHF d\Sigma-2\int_\Sigma\rho K Fd\Sigma.
	\end{equation}
\end{proposition}

\begin{proof}
	Using Propositions \ref{evolequations} and  \ref{evolvsupport},
	\begin{eqnarray*}
		\frac{d\Ic}{dt}& = & \int_\Sigma\frac{\partial \rho}{\partial t} Hd\Sigma+\int_\Sigma\rho \frac{\partial H}{\partial t}d\Sigma+\int_\Sigma\rho H\frac{\partial}{\partial t}d\Sigma\\
		& = & \int_\Sigma pHFd\Sigma
		+\int_\Sigma\rho\left(\Delta F+\left(|a|^2-(n-1)
		\right)F\right)d\Sigma -\\
		& & \qquad-\int_\Sigma\rho H^2Fd\Sigma\\
		& = & \int_MpHF d\Sigma+\int_\Sigma F\Delta\rho  d\Sigma+\int_\Sigma\rho\left(|a|^2-(n-1)\right)Fd\Sigma-\\
		& & \quad -\int_\Sigma\rho H^2f d\Sigma,
	\end{eqnarray*}
	and the result follows,  after some cancelations, from
	(\ref{extscalar}) and (\ref{laprho}).
\end{proof}

\begin{proposition}\label{suppevolv}
	Along the flow (\ref{evol}), the support function evolves according to
	\begin{equation}\label{suppevol}
		\frac{\partial p}{\partial t}=F\rho-\langle \nabla\rho,\nabla F\rangle.
	\end{equation}
	As a consequence,
	\begin{equation}\label{suppintevol}
		\frac{d}{dt}\int_\Sigma pd\Sigma=n\int_\Sigma F\rho d\Sigma.
	\end{equation}
\end{proposition}

\begin{proof}
	Using (\ref{suppdef}), (\ref{conf}) and (\ref{unit}) we compute:
	\begin{eqnarray*}
		\frac{\partial p}{\partial t} & = &  \frac{\partial}{\partial t}\langle D\rho,\xi\rangle\\
		& = & F\langle D_{\xi}D\rho,\xi\rangle + \langle D\rho,D_{\partial/\partial t}\xi\rangle\\
		& = & F\rho -\langle \nabla \rho,\nabla F\rangle,
	\end{eqnarray*}
	which proves (\ref{suppevol}). Now, using this and (\ref{evolvelemvol}),
	\begin{eqnarray*}
		\frac{d}{dt}\int_\Sigma pd\Sigma & = & \int_\Sigma \frac{\partial p}{\partial t}d\Sigma +
		\int_\Sigma p\frac{\partial}{\partial t}d\Sigma \\
		& = & \int_\Sigma F\rho d\Sigma -\int_\Sigma \langle \nabla \rho,\nabla F\rangle d\Sigma -
		\int_\Sigma pFH d\Sigma \\
		& = &  \int_\Sigma F\rho d\Sigma +\int_\Sigma F \Delta \rho d\Sigma -\int_\Sigma pFH d\Sigma,
	\end{eqnarray*}
	so that (\ref{suppintevol}) follows from (\ref{laprho}).
\end{proof}

The following proposition, proved in \cite{B}, plays a central role in our argument.

\begin{proposition}\label{brendle}
	If  $\Sigma\subset\mathbb H^n$ is  star-shaped and strictly mean convex then
	\begin{equation}\label{bre}
		(n-1)\int_\Sigma \frac{\rho}{H}d\Sigma\geq -\int_\Sigma pd\Sigma.
	\end{equation}
	Moreover, the equality holds if and only if $\Sigma$ is totally umbilical.
\end{proposition}

\begin{proof}
	This is a rather special case of the Heintze-Karcher-type inequality proved in \cite{B}, so we merely sketch the elegant argument there.
	The idea is to let $\Sigma$ flow under
	\begin{equation}\label{brendflow}
		\frac{\partial X}{\partial t}=\rho\xi,
	\end{equation}
	so we take  $F=\rho$ in (\ref{evol}). Using (\ref{evolvmean}), (\ref{evolvrho}), (\ref{laprho}) and (\ref{cauchy}) we see that, as long as the flow exists,
	\begin{eqnarray*}
		\frac{\partial}{\partial t}\frac{\rho}{H} & = & \frac{1}{H}\frac{\partial \rho}{\partial t}-
		\frac{\rho}{H^2}\frac{\partial H}{\partial t}\\
		& = & \frac{p\rho}{H}-\frac{\rho}{H^2}\left(\Delta \rho+\left(|a|^2-(n-1)\right)\rho\right)\\
		& = & -\frac{\rho^2}{H^2}|a|^2\\
		& \leq & -\frac{\rho^2}{n-1},
	\end{eqnarray*}
	so that by (\ref{evolvelemvol}),
	$$
	\frac{d}{dt}\int_\Sigma \frac{\rho}{H}d\Sigma\leq -\frac{n}{n-1}\int_\Sigma\rho^2d\Sigma.
	$$
	Combining this with (\ref{suppintevol}) we finally get
	$$
	\frac{d}{dt}\left((n-1)\int_\Sigma\frac{\rho}{H}d\Sigma+\int_\Sigma pd\Sigma\right)\leq 0,
	$$
	that is, the quantity within parenthesis is monotone non-increasing along the flow (\ref{brendflow}). The next step is to investigate the asymptotic behavior of solutions of (\ref{brendflow}). This might appear problematic at first sight but the key observation is that (\ref{brendflow}) is equivalent to the standard flow by {\em inward} parallel hypersurfaces ($F=1$) in the conformal metric
	$$
	\tilde g_1=\rho^{-2}g_1=\frac{dr^2}{\cosh^2 r}+\tanh^2r h.
	$$
	Thus, any solution becomes extinct in a certain finite time $t_*>0$ so that
	$$
	\lim_{t\to t_*}(n-1)\int_\Sigma\frac{\rho}{H}d\Sigma+\int_\Sigma pd\Sigma=0,
	$$
	as desired.
	In fact, an additional complication arises from the fact that the flow might develop singularities before the extinction time due to the appearance of cut points but, as explained in \cite{B}, a regularization procedure can be implemented to take care of this.
\end{proof}

\begin{remark}\label{brendlef}
	{\rm It follows from the computation above that
		$$
		\frac{\partial}{\partial t}\frac{H}{\rho}\geq \frac{H^2}{n-1},
		$$
		which implies that strict mean convexity is preserved under (\ref{brendflow}).}
\end{remark}

From now on we specialize to the flow
\begin{equation}\label{imcfreal}
	\frac{\partial X}{\partial t}=-\frac{\xi}{H},
\end{equation}
so that $F=-1/H$. This is the famous {\em inverse mean curvature flow},
which has been extensively studied in a variety of  contexts \cite{G1} \cite{U} \cite{HI} \cite{N}.
Here we will make use of  recent results by Gerhardt \cite{G2} \cite{G3} for evolving hypersurfaces in hyperbolic space, which we collect below.

\begin{proposition}\label{coll}
	If the initial hypersurface is star-shaped and strictly mean convex then the corresponding solution  is defined for all $t>0$ and expands the evolving hypersurfaces toward infinity while maintaining star-shapedness and strictly mean convexity.  Moreover, the hypersurfaces become strictly convex exponentially fast and also more and more umbilical in the sense that
	\begin{equation}\label{asymb}
		|b_i^j-\delta_i^j|\leq Ce^{-\frac{t}{n-1}}, \quad t>0,
	\end{equation}
	that is, the principal curvatures are uniformly bounded and converge exponentially fast to $1$.
	Moreover, there exists 
	$f:\mathbb S^{n-1}\to \mathbb R$ smooth so that, as $t\to +\infty$, the graphing function $u$ satisfies
	\begin{equation}\label{st5}
		\lim_{t\to +\infty}\left\|u-\frac{t}{n-1}-f(\theta)\right\|_{C^\infty(\mathbb S^{n-1})}=0.  
	\end{equation}
	In particular, 
	\begin{equation}\label{cruc2}
		\rho(u)=\cosh u=O(e^{\frac{t}{n-1}}), \quad \dot\rho(u)=\sinh u=O(e^{\frac{t}{n-1}}),
	\end{equation}
	and
	\begin{equation}\label{cruc3}
		|\nabla u|_h+|\nabla^2u|_h=O(1).
	\end{equation}
\end{proposition}

\begin{remark}
	\label{claimnot}
	{\em It is claimed in \cite{G2} that the function $f$ above is actually a constant, which means that the flow would deform the induced metric on the hypersurface to a round one after a suitable scaling. This is, however, not correct, as the concrete example in \cite{HW} shows. The correct asymptotics (\ref{st5})  appears in \cite{G3}.}   
\end{remark}

\section{The proofs of Theorems \ref{main} and \ref{main2}}\label{proofmain}

The proof of Theorem \ref{main} involves the consideration of two new monotone quantities along the solution of (\ref{imcfreal}) with $\Sigma$ as the initial hypersurface. Thus, for any closed $\Sigma\subset \mathbb H^n$ we set
\begin{equation}\label{renintsupp}
	\Jc(\Sigma)=-\int_\Sigma pd\Sigma,
\end{equation}
\begin{equation}\label{renarea}
	\Kc(\Sigma)=\omega_{n-1} \Ac(\Sigma)^{\frac{n}{n-1}},
\end{equation}
where $\Ac(\Sigma)=A/\omega_{n-1}$, and
\begin{equation}\label{funct}
	\Lc(\Sigma)=\frac{\Ic(\Sigma)-(n-1)\Kc(\Sigma)}{\Ac(\Sigma)^{\frac{n-2}{n-1}}}.
\end{equation}
To save notation, sometimes we write $\Ic(t)=\Ic(\Sigma_t)$, etc. As we shall see below, the new monotone quantities are $\Lc$ and $\Ac^{-\frac{n}{n-1}}(\Jc-\Kc)$.

\begin{proposition}\label{geode}
	On a geodesic sphere we have
	\begin{equation}\label{ic}
		\Lc\geq (n-1)\omega_{n-1}.
	\end{equation}
	Moreover, the equality holds if and only if the geodesic sphere is centered at the origin.
\end{proposition}

\begin{proof}
	If a geodesic sphere has  radius $r$ then its  area is $A=\omega_{n-1}\sinh^{n-1}r$ and its mean curvature is $H=(n-1)\coth r$. Furthermore, if it is centered at the origin then its support function is $p=-\sinh r$ by (\ref{supportform}). The equality in (\ref{ic}) then follows by a direct computation.
	On the other hand, if $\Sigma\subset\mathbb H^n$ is {\em any}  geodesic sphere of radius $r$ then
	$$
	K=\frac{(n-1)(n-2)}{2}\coth^2 r,
	$$
	so that (\ref{mink22}) yields
	\begin{eqnarray*}
		\int_\Sigma \rho Hd\Sigma & = & -\frac{2}{n-2}\int_\Sigma pKd\Sigma\\
		& = & -(n-1)\coth^2r\int_\Sigma pd\Sigma.
	\end{eqnarray*}
	Furthermore, if $B$ is the geodesic ball bounded by $\Sigma$,  (\ref{suppdef}), (\ref{conf}) and the divergence theorem imply
	\[
	\int_\Sigma pd\Sigma = -\int_B\Delta_{\mathbb H^n}\rho d\mathbb H^n = -n\int_B\rho d\mathbb H^n, \]
	so that
	\[
	\Ic(\Sigma)=\int_\Sigma \rho Hd\Sigma=n(n-1)\coth^2r\int_B \rho d\mathbb H^n.
	\] 
	It is clear from (\ref{funct}) and this way of writing $\Ic(\Sigma)$ as a volume integral involving $\rho$
	that the strict inequality in (\ref{ic}) holds 
	if $\Sigma$ is not centered at the origin.
\end{proof}

\begin{remark}\label{strict}
	{\rm Inequality (\ref{ic}) above just means that the inequality in Theorem \ref{main} holds for any geodesic sphere, with the equality occurring if and only if it is centered at the origin.
	}
\end{remark}

\begin{proposition}\label{inter}
	If the initial hypersurface $\Sigma$ in (\ref{imcfreal})  is star-shaped and strictly mean convex then
	\begin{equation}\label{evolva}
		\frac{d\Ac}{dt}=\Ac
	\end{equation}
	and
	\begin{equation}\label{evolvk}
		\frac{d\Kc}{dt}=\frac{n}{n-1}\Kc.
	\end{equation}
	Also,
	\begin{equation}\label{evolvj}
		\frac{d\Jc}{dt}\geq \frac{n}{n-1}\Jc,
	\end{equation}
	with the equality occurring if and only if $\Sigma$ is totally umbilical.
\end{proposition}

\begin{proof}
	The relation (\ref{evolvk}) follows from (\ref{evolva}), which is a consequence of (\ref{evolvarea}) with $F=-1/H$. Also, (\ref{evolvj}) follows immediately from (\ref{suppintevol}) and (\ref{bre}).
\end{proof}

The above result is crucial in establishing the existence of monotone quantities for the flow (\ref{imcfreal}).

\begin{proposition}\label{iii}
	If $\Sigma$ is star-shaped and strictly mean convex then
	\begin{equation}\label{jleqkkk}
		\frac{d}{dt}\frac{\Jc-\Kc}{\Ac^{\frac{n}{n-1}}}\geq 0,
	\end{equation}
	along any solution of (\ref{imcfreal}). Also, in any interval where $\Jc\leq \Kc$ there holds
	\begin{equation}\label{monoL}
		\frac{d\Lc}{dt}\leq 0.
	\end{equation}
	Moreover, if the equality holds in any of these inequalities for some $t$ then $\Sigma_{t}$ is totally umbilical.
\end{proposition}

\begin{proof}
	By (\ref{evolvk}) and (\ref{evolvj}) we get
	$$
	\frac{d}{dt}\left(\Jc-\Kc\right)\geq\frac{n}{n-1}\left(\Jc-\Kc\right),
	$$
	which by (\ref{evolva}) clearly yields (\ref{jleqkkk}).
	Moreover, by (\ref{evolrhoH}) with $F=-1/H$,
	$$
	\frac{d\Ic}{dt}=2\int_{\Sigma}\frac{\rho K}{H}d\Sigma + 2\Jc,
	$$
	so that by (\ref{lp}),
	$$
	\frac{d\Ic}{dt}\leq \frac{n-2}{n-1}\Ic + 2\Jc.
	$$
	From (\ref{evolvk}), after a rearrangement of terms, we get
	$$
	\frac{d}{dt}\left(\Ic-(n-1)\Kc\right)\leq \frac{n-2}{n-1}\left(\Ic-(n-1)\Kc\right)+2
	\left(\Jc-\Kc\right),
	$$
	which reduces to
	$$
	\frac{d}{dt}\left(\Ic-(n-1)\Kc\right)\leq \frac{n-2}{n-1}\left(\Ic-(n-1)\Kc\right),
	$$
	whenever $\Jc\leq \Kc$.
	In the presence of (\ref{evolva}), this immediately gives
	(\ref{monoL}). Finally, if the equality holds in either (\ref{jleqkkk}) or in (\ref{monoL}) then it holds in (\ref{bre}) as well.
\end{proof}

We start the proof of Theorem \ref{main} by noticing that
in
\cite[Theorem 1.1]{BHW} the authors establish a sharp geometric inequality for strictly mean convex, star-shaped hypersurfaces in the anti-deSitter-Schwarzschild space. By sending the mass parameter to zero, it follows from their work that if we set
\begin{equation}\label{bhwhyp2}
	\Mc=\frac{\Ic-(n-1)\Jc}{\Ac^{\frac{n-2}{n-1}}},
\end{equation}
then
there  holds
\begin{equation}\label{bhwhyp}
	\Mc(\Sigma)\geq (n-1)\omega_{n-1},
\end{equation}
for any $\Sigma\subset\mathbb H^n$ strictly mean convex and star-shaped,
with the equality holding if and only if $\Sigma$ is a geodesic sphere centered at the origin. Notice that this implies (\ref{afh}) whenever $\Jc(\Sigma)\geq \Kc(\Sigma)$, so we may assume that $\Jc(\Sigma)<\Kc(\Sigma)$.

We now let $\Sigma$ flow under (\ref{imcfreal}). In case $\Jc(\Sigma_t)>\Kc(\Sigma_t)$ for some $t>0$, let $t_0$ be the first value of the time parameter so that $\Jc(\Sigma_{t_0})=\Kc(\Sigma_{t_0})$.
Notice that $t_0$ exists because by (\ref{jleqkkk}) the quantity $\Ac^{-\frac{n}{n-1}}(\Jc-\Kc)$ is monotone nondecreasing along {\em any} solution of (\ref{imcfreal}).
Since $\Jc(\Sigma_{t})\leq \Kc(\Sigma_{t})$ for $t\leq t_0$, it then follows from Proposition \ref{iii} that
$\Lc(\Sigma)\geq \Lc(\Sigma_{t_0})=\Mc(\Sigma_{t_0})\geq (n-1)\omega_{n-1}$, where we used (\ref{bhwhyp}) in the last step. Thus, our main inequality (\ref{afh}) is also established in this case, so it remains to consider the case in which $\Jc(\Sigma_t)<\Kc(\Sigma_t)$ for any $t>0$. However, if this is the case then it follows again by Proposition \ref{iii} that $\Lc$ is monotone nonincreasing for all $t>0$.
But by Proposition \ref{refin}  we have
$$
\liminf_{t\to +\infty}\Lc(t)\geq(n-1)\omega_{n-1},
$$
so that
$$
\Lc(0)\geq (n-1)\omega_{n-1},
$$
which is just a rewriting of (\ref{afh}). Finally, we note that whenever the equality holds then it also holds in
(\ref{bre}), which implies that $\Sigma$ is a geodesic sphere necessarily centered at the origin by Remark \ref{strict}. This completes the proof of Theorem \ref{main}.

As remarked in the Introduction, the Penrose inequality in Theorem \ref{main2} follows immediately from (\ref{estima}) and the Alexandrov-Fenchel inequality (\ref{afh}) in Theorem \ref{main}. On the other hand, the rigidity statement follows from the arguments in \cite[Section 5]{DGS}. This completes the proof of Theorem \ref{main2}.

\appendix
\section{The asymptotic behavior of $\Lc$}\label{refined}

In this appendix we present a proof of the following proposition, which provides the expected limiting estimate for the quantity $\Lc$ along solutions of  the inverse mean curvature flow.
This asymptotic behavior is used in the proof of Theorem \ref{main}.

\begin{proposition}\label{refin}
	If $\Sigma_t$ is a solution of (\ref{imcfreal}) with $\Sigma_0$ strictly mean convex and star-shaped then
	\begin{equation}\label{ref}
		\liminf_{t\to+ \infty}\Lc(\Sigma_t)\geq (n-1)\omega_{n-1}.
	\end{equation}
\end{proposition}

We write the evolving hypersurfaces as graphs
of a function $u=u(t,\theta)$, $\theta\in\mathbb S^{n-1}$. Recall that $\rho(u)=\cosh u$
so that $\dot\rho(u)=\sinh u$ and
\begin{equation}\label{tata}
	\rho^2=\dot\rho^2+1.
\end{equation}
Also, if $v=\varphi(u)$, $\dot\varphi=1/\dot\rho$, as in (\ref{extrashi}), then
it follows from (\ref{cruc2}) and (\ref{cruc3}) that
\begin{equation}\label{analg2}
	|\nabla v|_h+|\nabla^2v|_h=O(e^{-\frac{t}{n-1}})
\end{equation}
and 
\begin{equation}
	\label{diff}
	|\rho(u)-\dot\rho(u)|=O(e^{-\frac{t}{n-1}}).
\end{equation}
Moreover, by (\ref{work2}),
\begin{equation}\label{fall}
	W^{-1}=1-\frac{1}{2}|\nabla v|^2_h+O(e^{-\frac{4t}{n-1}}).
\end{equation}

The induced metric is
\begin{equation}\label{4}
	g_{1{ij}}=\dot\rho^2(h_{ij}+v_iv_j),
\end{equation}
so that
\begin{equation}\label{spe}
	\sqrt{\det g_1}=\dot\rho^{n-1}\sqrt{\det h}\left(1+\frac{1}{2}|\nabla v|^2_h+O(e^{-\frac{4t}{n-1}})\right),
\end{equation}
so we get the expansions
\begin{equation}\label{eqarea}
	\Ac(\Sigma_t)=\fint \dot\rho^{n-1} +O(e^{\frac{(n-3)t}{n-1}}),
\end{equation}
and
\begin{equation}\label{eq2}
	\Ac(\Sigma_t)^{\frac{n-2}{n-1}}=\left(\fint \dot\rho^{n-1}\right)^{\frac{n-2}{n-1}}+
	O(e^{\frac{(n-4)t}{n-1}}),
\end{equation}
where
\[
\fint=\frac{1}{\omega_{n-1}}\int
\]
and the integration is over $\mathbb S^{n-1}$.

Recall that our intention is to estimate from below  the function
$$
\Lc(\Sigma_t)=\frac{\int_{\Sigma_t}\rho Hd\Sigma_t-(n-1)\omega_{n-1}\left(\Ac(\Sigma_t)\right)^{\frac{n}{n-1}}}
{\Ac(\Sigma_t)^{\frac{n-2}{n-1}}}.
$$
In terms of $v$, the second fundamental form of the evolving hypersurface is
\[
b_{ij}=\frac{\dot\rho}{W}\left(\rho(h_{ij}+v_iv_j)-(\nabla^2v)_{ij}\right).
\]
Notice also that by (\ref{4}) the inverse metric is
\[
g^{ij}_1=\dot\rho^{-2}\left(h^{ij}-\frac{v^iv^j}{W^2}\right),
\]
where $v^i=h^{ij}v_j$, so that the shape operator is
\[
a_j^i=g_1^{ik}b_{kj}=\frac{\rho}{W\dot\rho}\delta^i_j-\frac{1}{W\dot\rho}\tilde h^{ik}(\nabla^2 v)_{kj},
\]
where
\[
\tilde h^{ij}=h^{ij}-\frac{v^iv^j}{W^2},
\]
and from this we see that
\begin{equation}\label{newwtttt}
	\rho H=(n-1)W^{-1}\frac{\rho^2}{\dot\rho }-W^{-1}\frac{\rho}{\dot\rho}\Delta v+O(e^{-\frac{3t}{n-1}}).
\end{equation}
Thus, if we combine this with (\ref{spe}) and (\ref{fall})
we obtain
\begin{eqnarray*}
	\int_{\Sigma_t} \rho Hd\Sigma_t & = & (n-1)\int\rho^2\dot{\rho}^{n-2} -\int\dot{\rho}^{n-1}\Delta v + O(e^{\frac{(n-3)t}{n-1}})\\
	& = & (n-1)\int\rho^2\dot{\rho}^{n-2}+(n-1)\int \dot{\rho}^{n-2}\langle\nabla\dot{\rho},\nabla v\rangle_h
	+O(e^{\frac{(n-3)t}{n-1}})\\
	&= & (n-1)\int\rho^2\dot{\rho}^{n-2}+(n-1)\int \rho^2\dot{\rho}^{n-2}|\nabla v|_h^2 +O(e^{\frac{(n-3)t}{n-1}}).
\end{eqnarray*}

On the other hand, by H\"older inequality and (\ref{spe})
we find that 
\begin{eqnarray*}
	\omega_{n-1} \Ac(\Sigma_t)^{\frac{n}{n-1}} &\leq &\int (\sqrt{{\rm det}\,g_1})^{\frac{n}{n-1}}\\
	& = & \int \dot\rho^n+\frac{n}{2(n-1)}\int \dot{\rho}^n|\nabla v|_h^2 +O(e^{\frac{n-4}{n-1}t}). 
\end{eqnarray*}
Thus, by (\ref{tata}) we obtain 
\begin{eqnarray*}
	\int_{\Sigma_t} \rho Hd\Sigma_t -(n-1)\omega_{n-1}\Ac(\Sigma_t)^{\frac{n}{n-1}}
	& \geq & 
	(n-1)\int\dot{\rho}^{n-2}+\frac{n-2}{2}\int \dot{\rho}^{n}|\nabla v|_h^2\\
	& & \quad + O(e^{\frac{n-3}{n-1}t})\\
	& = & (n-1)\int\dot{\rho}^{n-2} \\
	& & +\frac{n-2}{2}\int \dot{\rho}^{n-4}|\nabla \dot{\rho}|_h^2
	+ O(e^{\frac{n-3}{n-1}t}),
\end{eqnarray*}
so if we take (\ref{eq2}) into account we see that proving  (\ref{ref}) amounts to checking that 
\begin{equation}\label{beckner0}
	(n-1)\int\dot{\rho}^{n-2}+\frac{n-2}{2}\int \dot{\rho}^{n-4}|\nabla \dot{\rho}|_h^2\geq (n-1)\omega_{n-1}^{\frac{1}{n-1}}\left(\int \dot{\rho}^{n-1}\right)^{\frac{n-2}{n-1}}.
\end{equation}
But, as observed in \cite{BHW}, this is an immediate consequence of a sharp Sobolev type inequality by Beckner \cite{Be}. This completes the proof of Proposition \ref{refin}.

\section{The asymptotically locally hyperbolic case}
\label{alhcaseap}

In this appendix we briefly describe how the argument leading to Theorem \ref{main2} can be easily adapted to recover its generalization given by Theorem \ref{theomain}. 

Clearly, the key step is to prove 
the Alexandrov-Fenchel-type inequality (\ref{aftype2}) in Theorem \ref{aftype}. We start by observing that Propositions \ref{evolvsupport}, \ref{evolvfunc} and \ref{suppevolv} remain true with $\rho$ replaced by $\rho_\epsilon$, since their proofs only uses that the ambient manifold is locally hyperbolic and carries the conformal field $\nabla_{g_\epsilon}\rho_\epsilon$; see (\ref{conf}). Also, since this ambient manifold satisfies the structural conditions in the main result in \cite{B}, the analogue of Proposition \ref{brendle} also holds true. Taken together, these facts imply that the analogue of Proposition \ref{iii} still holds true, so that the proof of (\ref{aftype2}) boils down to checking that 
\begin{equation}\label{refeps}
	\liminf_{t\to+\infty}\Lc(\Sigma_t)\geq (n-1)\vartheta_{n-1}\epsilon,
\end{equation}  
where $\Sigma_t$ is the solution to the inverse  mean curvature flow having $\Sigma$ as initial hypersurface;
compare with (\ref{ref}). We note that the left-hand side of (\ref{refeps}) makes sense because the analogue of Proposition \ref{coll} remains true, which follows straightforwardly from the methods in \cite{G2} \cite{G3}. Taking into account that (\ref{tata}) should be replaced by $\rho^2=\dot\rho^2+\epsilon$, we see that (\ref{refeps}) reduces to proving that
\[
(n-1)\epsilon\int_N\dot{\rho}^{n-2}+\frac{n-2}{2}\int_N \dot{\rho}^{n-4}|\nabla \dot{\rho}|_h^2\geq (n-1)\epsilon\omega_{n-1}^{\frac{1}{n-1}}\left(\int_N \dot{\rho}^{n-1}\right)^{\frac{n-2}{n-1}};
\]
compare to (\ref{beckner0}).
Since  the validity of this inequality is immediate for $\epsilon=0,-1$, and the rigidity statement follows from a simple adaptation of the arguments in \cite[Section 5]{DGS}, the proof of Theorem \ref{theomain} is completed.

\end{document}